\documentclass[reqno]{amsart}
\usepackage{amssymb}
\usepackage{amsmath,amscd}
\usepackage{graphicx}
\usepackage{color,hyperref}
\usepackage{tikz}
\usetikzlibrary{matrix,arrows}
\usepackage{enumitem}

\newcommand{\mb}{\mathbf}

\newcommand{\C}{\mathbb C}

\newcommand{\N}{\mathbb N}

\numberwithin{equation}{section}

\theoremstyle{plain}
\newtheorem{theorem}{Theorem}
\newtheorem{lemma}{Lemma}

\theoremstyle{definition}
\newtheorem{example}{Example}

\begin{document}
	
\title[Generalized Chebyshev Acceleration]{Generalized Chebyshev Acceleration \\on the unit disc}
	
\author[Nurgül Gökgöz]{Nurgül Gökgöz}
\address{Çankaya University, Mathematics Department, 06815 Ankara, Turkey.}
%\curraddr{}
\email{ngokgoz@cankaya.edu.tr}

\date{\today}
	
\begin{abstract}
Generalized Chebyshev acceleration is a semi-iterative technique applicable to a basic iterative method only when the eigenvalues of the iteration matrix satisfy a highly restrictive inclusion condition. In this work, we relax this requirement by introducing an alternative iterative scheme that converges to the same solution. The effectiveness of the proposed approach is examined through its application to a large-scale sparse normal matrix.
\end{abstract}
	
%\dedicatory{In memory of }
	
	
	
\maketitle
\section{Introduction}
Iterative methods can be useful for solving systems of linear equations $A\mb{x}=\mb{b}$ involving many variables where direct methods exceed the available computing power. Basic examples of stationary iterative methods use a splitting of the matrix $A$ by means of its diagonal and triangular (upper and lower) parts. See \cite{greenbaum}, and \cite{varga}, for several iterative algorithms.

In this paper we focus on the following iterative process:
\begin{equation}\label{iter}
	\mb{x}^{(m)}=M\mb{x}^{(m-1)}+\mb{g}
\end{equation}
We assume that the spectral radius $\rho(M)$ is less than one so that this process converges. For each initial vector $\mb{x}^0$, such an iterative procedure converges to the unique solution of $(I_n-M)\mb{x}=\mb{g}$.

Chebyshev acceleration is a semi-iterative method that works efficiently if $M$ has real spectrum \cite{varga}. Recently, we have introduced a new semi-iterative method which works for a larger spectrum \cite{gokgoz}. This generalization is possible thanks to the existence of generalized Chebyshev polynomials associated with the root system $A_2$. This family of polynomials has very nice dynamical properties, and maps a large subset of the unit disc onto itself, which turns out to be a deltoid region \cite{veselov, veselov-survey}. This analogy, with the closed interval $[-1,1]$ for classical Chebyshev polynomials, is the key property for introducing a generalized semi-iterative method inspired by the classical Chebyshev acceleration.

This new semi-iterative method has a serious restriction for its application to (\ref{iter}). Suppose that we know a nonzero eigenvalue $\lambda_1$ of $M$ with maximal absolute value, i.e. $0\neq |\lambda_1|=\rho(M)$. To apply the generalized Chebyshev acceleration, the quotients $\lambda/\lambda_1$ must lie in a certain region $\Delta$, defined by (\ref{deltoid}), for each eigenvalue $\lambda$ of $M$. However, real-life problems yield iterative methods whose associated matrices $M$ do not generally satisfy this condition.

In this paper, we extend the use of the generalized Chebyshev acceleration, associated with $A_2$, from the deltoid $\Delta$ to the whole unit disc. This is achieved by replacing (\ref{iter}) with an alternative iterative process that converges to the same solution. More precisely, we consider
\begin{equation}\label{newiter}
	\mb{x}^{(m)}=M^k\mb{x}^{(m-1)}+\mb{h}
\end{equation}
where  $\mb{h} = (I+M+\ldots + M^{k-1})\mb{g}$. This new iterative process has a different iteration matrix, namely $M^k$, with eigenvalues in a smaller disc rather than the whole unit disc. If $k\geq 2$, we need to perform more multiplications of vectors by $M$, and therefore each iterative step requires more computation. On the other hand, this extended method is applicable to almost all iterative processes (\ref{iter}). A case-by-case analysis can be found in Theorem~\ref{existingk} and Theorem~\ref{exclude}.

The organization of the paper is as follows. In the second section, we summarize the basics of semi-iterative methods, fix some notation, and outline the classical and the generalized Chebyshev acceleration methods. In the third section, we explain how to relax the membership condition for the quotients of the eigenvalues of $M$ from the deltoid to the unit disc, and give a low-dimensional example to illustrate that this approach may not be optimal in certain cases. In the last section, we analyze the efficiency of our method with sparse matrices of large dimension.

\section{Semi-Iterative Methods}
An iterative method is a mathematical procedure that uses an initial value to generate a sequence of improving approximate solutions. For a comprehensive treatment of the subject, one can see \cite{greenbaum}, and \cite{varga}.

We start with summarizing the basics of semi-iterative methods, and fix some notation that will be used in the rest of the paper. We will mostly follow the exposition of \cite{varga}.

Consider the iterative process (\ref{iter}), which convergences to the required solution of the system $(I-M)\mb{x}=\mb{g}$. This unique solution satisfies
\begin{equation}\label{relation}
	\mb{x}=M\mb{x}+\mb{g}.
\end{equation}
We define the error vector as $\epsilon^{(m)}:=\mb{x}-\mb{x}^{(m)}$. Combining (\ref{iter}) and (\ref{relation}) inductively, we see that the $m$th error vector equals $M^m\epsilon^{(0)}$. The error vector converges to zero if and only if all the eigenvalues $M$ have absolute values less than one, i.e., the spectral radius $\rho(M)$ is less than one. 

We may obtain a faster rate of convergence by weighting the components of the error vector suitably. Consider the following linear combination
\begin{equation}\label{y}
	\mb{y}^{(m)} := \sum_{j=0}^{m} \nu_j(m)\mb{x}^{(j)}, \quad m\geq 0.
\end{equation}
This linear combination is called a \textit{semi-iterative method} with respect to the iterative method of (\ref{iter}). We want this new iteration to converge to the same solution, and therefore, there is a restriction on the coefficients $\nu_j(m)$. We must have
\begin{equation}\label{restrict}
	\sum_{j=0}^{m} \nu_j(m)=1.
\end{equation}
We shall relate the error of the former iterative process to the error of this new setup. For this purpose, we introduce a new notation, and define the error vector in the semi-iterative case to be $\eta^{(m)}:=\mb{x}-\mb{y}^{(m)}$.  Using the linear combination given in the equation (\ref{y}), we obtain 
\begin{equation*}
	\eta^{(m)} =\mb{x}-\sum_{j=0}^{m} \nu_j(m)\mb{x}^{(j)} = \sum_{j=0}^{m} \nu_j(m)(\mb{x}-\mb{x}^{(j)}) = \sum_{j=0}^{m} \nu_j(m)\epsilon^{(j)}.
\end{equation*}
We define $p_m(\lambda) := \sum_{j=0}^{m} \nu_j(m)\lambda^j$. It follows that 
\begin{equation}\label{eta}
	\eta^{(m)} = p_m(M)\epsilon^{(0)}.
\end{equation}
With standard definitions of vector and matrix norms, see \cite{varga}, it follows that 
\begin{equation*}
	\lVert\eta^{(m)}\rVert = \lVert  p_m(M)\epsilon^{(0)} \rVert \leq \lVert  p_m(M)\rVert \cdot \lVert\epsilon^{(0)} \rVert, \quad m\geq 0.
\end{equation*}
If we can make the quantities $\lVert p_m(M)\rVert$ smaller, compared to $\lVert M^m\rVert$, then the bounds for $\lVert\eta^{(m)} \rVert$ for the semi-iterative method are proportionally smaller, compared to $\lVert\epsilon^{(m)}\rVert$.

\subsection{Chebyshev Acceleration}
Now, we give a short summary of the Chebyshev semi-iterative method with respect to (\ref{iter}) following \cite{varga}. Suppose that $M$ has real eigenvalues with absolute values strictly less than one, and suppose also that we know the spectral radius $0<\rho=\rho(M)<1$. By the definition of Chebyshev semi-iterative method, we choose
\begin{equation}\label{pm}
	p_m(\lambda)=\frac{C_m(\lambda/\rho)}{C_m(1/\rho)}.
\end{equation}
where $C_m$ is the $m$th Chebyshev polynomial. We obviously have $p_m(1)=1$. In other words the restriction (\ref{restrict}) is satisfied. The Chebyshev polynomials $C_m$ are uniquely defined by the following functional equation
\[C_m(\cos(\theta))=\cos(m\theta).\]
The first few Chebyshev polynomials are $C_0(t)=1$, $C_1(t)=t$, $C_2(t)=2t^2-1$, and they satisfy the recurrence relation 
\[C_{m}(t)=2tC_{m-1}(t)-C_{m-2}(t), \quad m\geq 2.\]
This relation, applied to $C_m(\lambda/\rho) = C_m(1/\rho)p_m(\lambda)$, gives us
\[p_m(\lambda) = \frac{2\lambda C_{m-1}(1/\rho)}{\rho C_m(1/\rho)}p_{m-1}(\lambda) - \frac{C_{m-2}(1/\rho)}{C_m(1/\rho)}p_{m-2}(\lambda), \quad m \geq 2.\]
We now plug in $M$ for $\lambda$, and postmultiply throughout by $\epsilon^{(0)}$. Using the relation with the error vector, see (\ref{eta}), and the definition $\eta^{(m)}=\mb{x}-\mb{y}^{(m)}$,
we obtain the following semi-iterative process
\begin{equation}\label{itercheby}
	\mb{y}^{(m)} = \frac{2C_{m-1}(1/\rho)}{\rho C_{m}(1/\rho)}(M\mb{y}^{(m-1)}+\mb{g}) - \frac{C_{m-2}(1/\rho)}{C_{m}(1/\rho)}\mb{y}^{(m-2)}.
\end{equation}
For simplicity, we can choose $\mb{y}^{(0)}=\mb{x}^{(0)}$, and  $\mb{y}^{(1)}=\mb{x}^{(1)}$. Unlike (\ref{iter}), this semi-iterative method includes two earlier vectors instead of one. This requires additional vector storage. This is of minor importance in computer applications.  We also note that, the terms $\mb{y}^{(j)}$ for $j=0,1,2, \ldots$ can be computed directly without going through (\ref{iter}).

\subsection{Generalized Chebyshev Acceleration}
Lastly, we give a short summary of the generalized Chebyshev semi-iterative method, associated with the root system $A_2$, with respect to (\ref{iter}). See \cite{gokgoz} for more details.

The generalized Chebyshev polynomials are defined by means of exponential invariants of Bourbaki \cite{Bourbaki}. These polynomials are studied by Veselov \cite{veselov}, and by Hoffman and Withers \cite{hoffwith}, independently, and they have interesting properties in the theory of dynamical systems and numerical analysis. We will use a normalized generalized cosine function given in \cite{munthe}. For $\theta=(\theta_1,\theta_2) \in \mathbb{C}^2$, we define $\Phi=(\phi_1,\phi_2)$ as follows
\begin{equation}\label{gencos}
	\varphi_1(\theta)=\frac{1}{3}\left(e^{2\pi i \theta_1}+e^{-2\pi i \theta_2}+e^{2\pi i (\theta_2-\theta_1)}\right), \quad \varphi_2(\theta)=\overline{\varphi_1(\theta)}.
\end{equation}
There exists an associated infinite sequence of polynomial mappings $\mathcal{A}_m: \C^2 \rightarrow \C^2, m \in {\N}$ satisfying the conditions
\begin{equation*}
	\mathcal{A}_m(\Phi(\theta)) = \Phi(m\theta).
\end{equation*}
If we write $\mathcal{A}_m=(f_m,g_m)$, then the first component satisfies $f_m(\Phi(\theta))=\varphi_1(m\theta)$. Putting $x=\varphi_1(\theta)$, and $\bar{x}=\varphi_2(\theta)$, and using (\ref{gencos}), we see that the second component is obtained by switching the variables of the first component, i.e. $f_m(x,\bar{x})=g_m(\bar{x},x)$. In this paper, we will use the first component, and we write $f_m(x)$ instead of $f_m(x,\bar{x})$, for simplicity. The first few of these polynomials are listed below:
\begin{align*}
	f_0(x) &= 1\\
	f_1(x) &= x\\
	f_2(x) &= 3x^2-2\bar{x}\\
	f_3(x) &= 9x^3-9x\bar{x}+1
\end{align*}
These polynomials satisfy the following recursive relation
\begin{equation}\label{recursion}
	f_m = 3xf_{m-1}-3\bar{x}f_{m-2}+f_{m-3}, \quad m\geq 3.
\end{equation}
The key property of these polynomials that motivates the work \cite{gokgoz} is the fact that $f_m$ maps a large subset of $\C$ onto itself. Consider
\begin{equation}\label{deltoid}
	\Delta=\{ \varphi_1(\theta) \mathrel{:} \theta \in \mathbb{R}^2\}.
\end{equation}
This region is enclosed by \textit{Steiner hypocycloid}, a deltoid curve with three corners at third roots of unity. For an illustration of this region, see Figure~\ref{fig:deltoid}.
\begin{figure}[htbp]
	\centering
	\includegraphics[scale=0.6]{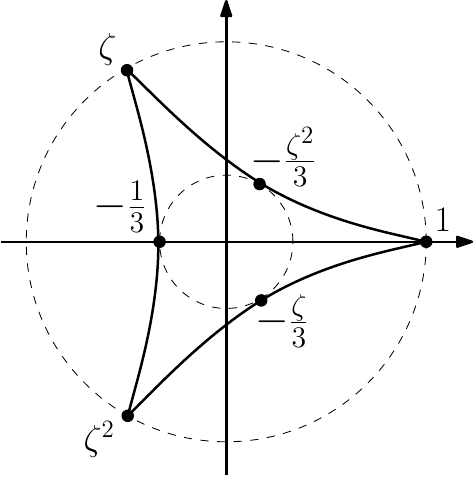}
	\caption{The region $\Delta$.}
	\label{fig:deltoid}
\end{figure}	

The Julia set of the polynomials $\mathcal{A}_m=(f_m,g_m)$ is $\Delta\times \Delta$. See \cite{veselov} or \cite{veselov-survey}. So the region $\Delta$ is analogous to the interval $[-1,1]$ in the classical case. If we have a real number $\gamma\in [-1,1]$, then the sequence $\{C_m(\gamma)\}_{m=1}^\infty$ remains in the interval $[-1,1]$. Similarly, if we have a complex number $\tau\in\Delta$, then the sequence $\{f_m(\tau)\}_{m=1}^\infty$ remains in the region $\Delta$. This analogy allows us to generalize the Chebyshev semi-iterative method to a wider setting. Following (\ref{pm}), and adapting to this new situation, we choose
\[p_m(\lambda) = \frac{f_m(\lambda/\lambda_1)}{f_m(1/\lambda_1)},\]
where $\lambda_1$ is an eigenvalue of $M$ with maximal absolute value. We obviously have $p_m(1)=1$, and the basic restriction (\ref{restrict}) is satisfied. Moreover, using the recursive relation (\ref{recursion}), we now have
\begin{align}\label{keyrecursion}
	\begin{split}	
		f_m(1/\lambda_1)p_m(\lambda) =\ & 3 \frac{\lambda}{\lambda_1} f_{m-1}(1/\lambda_1)p_{m-1}(\lambda)\\
		&-3\frac{\bar{\lambda}}{\bar{\lambda}_1} f_{m-2}(1/\lambda_1)p_{m-2}(\lambda)\\
		&+f_{m-3}(1/\lambda_1)p_{m-3}(\lambda), \quad m\geq 3.
	\end{split}	
\end{align}
We want the condition (\ref{eta}) to be satisfied. In order to relate the term $\bar{\lambda}$ in the above equation to the matrix $M$, while forming our new iterative method, we need to introduce a new matrix that is closely related to $M$.

Suppose that $M$ has spectral radius $0<\rho(M)<1$ with possibly non-real eigenvalues. Suppose that we know an eigenvalue $\lambda_1$ with maximal absolute value, i.e. $0<|\lambda_1|=\rho(M)$. Suppose further that $\lambda/\lambda_1$ lies in the region $\Delta$, defined by (\ref{deltoid}), for each eigenvalue $\lambda$ of $M$. Suppose $\widetilde{M}$ is an $n\times n$ matrix with the same set of eigenvalues as $M$, but the eigenvectors are switched with their complex conjugate counterparts. If $M$ is a normal matrix, then $\widetilde{M}$ can be chosen to be the conjugate transpose of $M$. More generally, if $D_0=P^{-1}MP$ is a diagonal matrix, then one can pick $\widetilde{M} = P \overline{D}_0 P^{-1}$. We also suppose that $\widetilde{M}\mb{x} + \widetilde{\mb{g}}=\mb{x}$ for some vector $\widetilde{\mb{g}}$. This condition allows us to deal with complex conjugation occurring in (\ref{keyrecursion}), and to have (\ref{eta}) to be satisfied for this particular $p_m$. 

We already have $M\mb{x}+\mb{g}=\mb{x}$ in the original setup. In addition, we have assumed  $\widetilde{M}\mb{x} + \widetilde{\mb{g}}=\mb{x}$ for some vector $\widetilde{\mb{g}}$. Let us plug in $M$, and $\widetilde{M}$ for $\lambda$, and $\bar{\lambda}$, respectively, in (\ref{keyrecursion}), and postmultiply throughout by $\epsilon^{(0)}$. Using the relation with the error vector, see (\ref{eta}), and by the choice of $\widetilde{M}$ associated with $M$, we obtain the following semi-iterative method,
\begin{align}\label{main}
	\begin{split}		
		\mb{y}^{(m)}=\ &\frac{3f_{m-1}(1/\lambda_1)}{\lambda_1 f_m(1/\lambda_1)}\left( M\mb{y}^{(m-1)} + \mb{g} \right) \\
		& -\frac{3f_{m-2}(1/\lambda_1)}{\bar{\lambda}_1 f_m(1/\lambda_1)}\left( \widetilde{M}\mb{y}^{(m-2)} + \widetilde{\mb{g}} \right)\\
		& +\frac{f_{m-3}(1/\lambda_1)}{f_m(1/\lambda_1)}\mb{y}^{(m-3)}, \quad m\geq 3.
	\end{split}		
\end{align}
For simplicity, we choose $\mb{y}^{(0)}=\mb{x}^{(0)}$, $\mb{y}^{(1)}=\mb{x}^{(1)}$, and $\mb{y}^{(2)}=\mb{x}^{(2)}$. This new semi-iterative method is of a similar form to (\ref{itercheby}), but involves three earlier vectors instead of two. This requires additional vector storage. Moreover, an additional matrix multiplication exists compared to the classical Chebyshev acceleration. We will take this into consideration in the last section while analyzing the efficiency of the generalized Chebyshev acceleration when it is applied to a normal sparse matrix of large dimension.

\section{From Deltoid to Unit Disc.}
An iterative process, such as (\ref{iter}), naturally comes with a matrix with spectral radius $\rho(M)$ less than one. Suppose that we know an eigenvalue $\lambda_1$, that is nonzero, with maximal absolute value, i.e. $|\lambda_1|=\rho(M)$. Recall that there are two restrictions while applying the generalized Chebyshev semi-iterative method, as listed below:
\begin{enumerate}
	\item the knowledge of a nonzero eigenvalue $\lambda_1$ with maximal absolute value, i.e. $0<|\lambda_1|=\rho(M)$, and the property that $\lambda/\lambda_1$ lies in the region $\Delta$ for each eigenvalue $\lambda$ of $M$.	
	\item the knowledge of an associated matrix $\widetilde{M}$ with the same set of eigenvalues as $M$ but the eigenvectors are switched with their complex conjugate counterparts, and an associated vector $\widetilde{\mb{g}}$ such that $(I_n-\widetilde{M})\mathbf{x}=\widetilde{\mb{g}}$.
\end{enumerate}
In this section, we will explain how the first restriction can be removed in most cases with an alternative iterative process (\ref{newiter}) that possibly requires more iterations. We start with proving the equivalence of these iterative methods.

\begin{lemma} Let $M$ be a matrix with spectral radius $\rho(M)<1$. Then the iterative processes (\ref{iter}) and (\ref{newiter}) converge to the same vector. 
\end{lemma}
\begin{proof}
Consider the polynomial $$p(t)=1+t+\ldots + t^{n-1}=\frac{t^n-1}{t-1}.$$ 
Each root of $p(t)$ is an $n$th root of unity, and it has absolute value one. The eigenvalues of the matrix $p(M)=I+M+\ldots + M^{k-1}$ is in one-to-one correspondence with the eigenvalues of $M$. Indeed, if $\lambda$ is an eigenvalue of $M$, then $p(\lambda)$ is an eigenvalue of $p(M)$. Since the spectral radius of $M$ is strictly less than one, the value $p(\lambda)$ is never zero. Therefore the matrix $p(M)$ is invertible. 

Recall that the original iterative process (\ref{iter}) converges to the unique solution of $(I-M)\mb{x}=\mb{g}$. Clearly, the spectral radius of $M^k$ is less than one. We observe that the new iterative process (\ref{newiter}) converges to the unique solution of $(I-M^k)\mb{x}=\mb{h}$ where $\mb{h}=p(M)\mb{g}$. These solutions are the same since 
$$(I-M^k)^{-1}\mb{h} = (I-M)^{-1}p(M)^{-1}\mb{h} = (I-M)^{-1}\mb{g}.$$
This finishes the proof.
\end{proof}

Now we investigate how the integer $k$ may be chosen. We study this problem in two separate cases that can be related to the spectral radius. For this purpose we recall that an eigenvalue $\lambda$ of $M$ is called \textit{dominant} if $\lambda=\rho(M)$. The importance of dominant eigenvalues comes from the fact that the Chebyshev acceleration, both the classical and the generalized version, uses a scaling of the spectrum.

Let $\lambda_1$ be a dominant nonzero eigenvalue of $M$. If the iteration matrix $M$ has a real spectrum and then the quotients $\lambda/\lambda_1$ naturally remain in the interval $[-1,1]$ for each eigenvalue $\lambda$ of $M$. This makes it possible to apply the classical Chebyshev acceleration method. However, in the generalized Chebyshev acceleration method, the quotients $\lambda/\lambda_1$ may not remain in the deltoid $\Delta$. 

Before we go further, we make a simple observation. Suppose that we want to apply the generalized Chebyshev acceleration to the iterative method (\ref{newiter}) with $k\geq 2$. We shall have $(\lambda/\lambda_1)^k\in\Delta$. This is possible if and only if the quotient $\lambda/\lambda_1$ is included in the inverse image of the deltoid $\Delta$ under the power map $z\mapsto z^k$. This inverse image gets larger and larger as $k$ gets bigger. For an illustration of this phenomenon, see Figure~\ref{fig:invimg}.
\begin{figure}[htbp]
	\centering
	\includegraphics[scale=0.8]{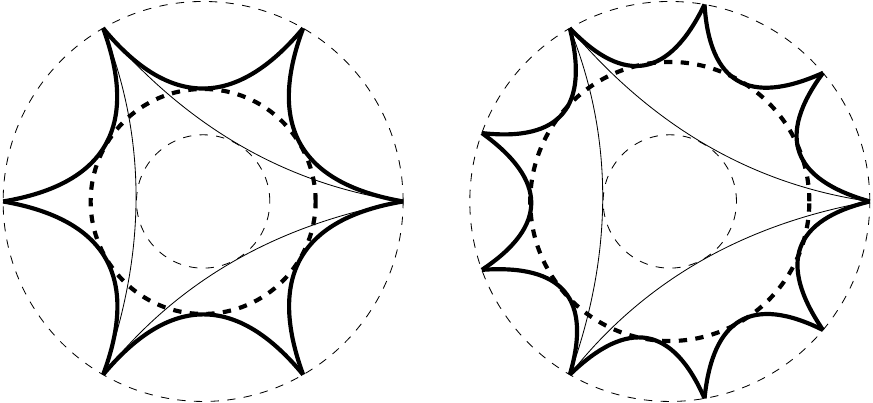}
	\caption{The inverse images of $\Delta$ under $z\mapsto z^k$ for $k=2,3$, respectively.}
	\label{fig:invimg}
\end{figure}

\subsection{Unique dominant eigenvalue}
Consider the iterative method (\ref{iter}), and assume that it has a unique dominant eigenvalue $\lambda_1$. This property can be checked in certain cases by using the Perron–Frobenius theorem, which asserts that a real square matrix with positive entries has a unique real dominant eigenvalue \cite{meyer}. In general, especially for arbitrarily large sparse matrices, there is no such conclusive result. However, one may use iterative methods, such as Arnoldi's iteration, to investigate the eigenvalues whose absolute values are large \cite{arnoldi}. 

Now we give our first main result.
\begin{theorem}\label{existingk}
Let $M$ be a matrix with spectral radius $\rho(M)<1$, and suppose that $M$ has a unique dominant eigenvalue $\lambda_1$. Let $k$ be the smallest positive integer such that $$\frac{1}{\sqrt[k]{3}} \geq \left|\frac{\lambda}{\lambda_1}\right|$$ for all eigenvalues of $\lambda$ of M. Then the generalized Chebyshev acceleration is applicable to (\ref{newiter}) for all positive integers bigger than or equal to $k$.
\end{theorem}
\begin{proof}
The boundary of the deltoid $\Delta$ is given by $z=(2e^{it}+e^{-2it})/3$. The triangle inequality implies that $\Delta$ contains the closed disc of radius $1/3$. See Figure~\ref{fig:deltoid} for an illustration of this inequality. 

Suppose that $\lambda_2$  is an eigenvalue of $M$ satisfying
\[\rho(M) = |\lambda_1| > |\lambda_2| \geq |\lambda|	\]
for all eigenvalues $\lambda$ of $M$ except $\lambda_1$. If $\lambda/\lambda_1 \in \Delta$ for each eigenvalue $\lambda$ then we can choose $k=1$, apply the generalized Chebyshev acceleration directly to (\ref{iter}).

Assume otherwise. The complex number $\lambda_2/\lambda_1$ has absolute value less than one. On the other hand, the quantity $1/\sqrt[k]{3}$ approaches $1$ as $k$ goes infinity. Let $k$ be the smallest positive integer such that  $1/\sqrt[k]{3} \geq |\lambda_2/\lambda_1|$. It follows that  
$$ \left|\frac{\lambda^k}{\lambda_1^k}\right| \leq \left|\frac{\lambda_2^k}{\lambda_1^k}\right| \leq \frac{1}{3}.$$
In other words, the quotients $\lambda^k/\lambda_1^k$ is included in $\Delta$ for all eigenvalues $\lambda$ of $M$. This means that the generalized Chebyshev acceleration is applicable to (\ref{newiter}) with this choice of $k$.
\end{proof}

We remark that if the absolute value of $\lambda_2/\lambda_1$, as it is defined in the above proof, is small, then the generalized Chebyshev acceleration has a great potential. In such a case, it is guaranteed by Theorem~\ref{existingk} that a rather small $k$ can be used. See the following list of numbers in this respect:
\begin{equation}\label{k-ratios}
	\begin{array}{|c|c|c|c|c|c|c|c|}\hline
		k&1&2&3&4&5&6&7 \\ \hline
		1/\sqrt[k]{3}&0.333&0.577&0.693&0.759&0.802&0.832&0.854\\ \hline
	\end{array}
\end{equation}
The hypothesis of Theorem~\ref{existingk} is quite general and covers many cases. On the other hand, we remark that the $k$-value given by this theorem may not be the optimal choice. The positions of the quotients $\lambda/\lambda_1$ may allow us to use a smaller $k$-value. We illustrate this with an example. 

\begin{example}\label{ex1}
Consider the iterative method (\ref{iter}) with the four-by-four iteration matrix defined as follows
\[M=\begin{bmatrix}1.40 + 0.70i&-1.80- 2.80i&1.20- 2.80i&0.20+0.00i\\ 0.25 + 0.35i&-0.95 - 1.05i&-0.60 - 0.70i&-0.85 + 0.35i\\ 0.00+0.00i&0.90 + 0.70i&1.30 + 1.40i&0.90 + 0.70i\\ -0.25 - 0.35i&-0.45+ 0.35i&-1.20 - 0.70i&-0.55 - 1.05i
\end{bmatrix}.\]
This iteration matrix $M$ has eigenvalues $\lambda_1 = 0.90$, $\lambda_{2,3} = 0.40 \pm 0.70i$, and $\lambda_4 = -0.50$. Note that $\rho(M)=|\lambda_1| > |\lambda_2|=|\lambda_3|>|\lambda_4|$. The quotient $\lambda_2/\lambda_1$ is not in $\Delta$. As a result, generalized Chebyshev acceleration is not directly applicable (\ref{iter}). However, the quotients $\lambda_i/\lambda_1$ are in the inverse image of $\Delta$ under the power map $z\mapsto z^2$. See Figure~\ref{fig:eigenvalues}. Thus, we can apply the generalized Chebyshev acceleration to (\ref{newiter}) with $k=2$. 

\begin{figure}[htbp]
	\centering
	\includegraphics[scale=0.8]{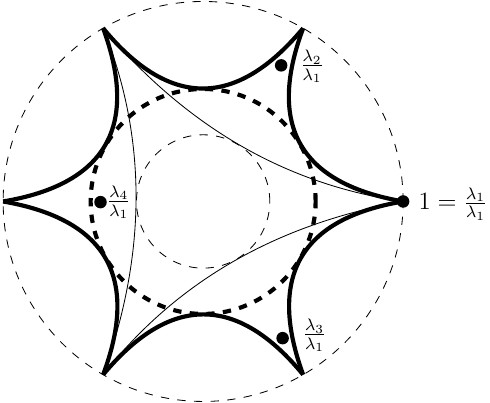}
	\caption{The positions of the values $\frac{\lambda}{\lambda_1}.$}
	\label{fig:eigenvalues}
\end{figure}

We observe that the $k$-value provided by Theorem~\ref{existingk} is far from being optimal. The eigenvalues $\lambda_1$ and $\lambda_2$ have absolute values that are very close to each other. That's why we need a big exponent $k$ so that  $(\lambda_2/\lambda_1)^k$ is included in the closed disc $|z|\leq 1/3$. Indeed, we have the following equality
$$\frac{1}{\sqrt[9]{3}}< \left|\frac{\lambda_2}{\lambda_1}\right| = 0.895 < \frac{1}{\sqrt[10]{3}}.$$
If we use Theorem~\ref{existingk}, without using the extra information about the positions of the quotients $\lambda/\lambda_1$, then we shall pick $k=10$. This requires a huge number of new iterations to be computed in comparison to the original choice $k=2$. To apply the generalized Chebyshev acceleration effectively via (\ref{newiter}), we remark that either the quotients $\lambda/\lambda_1$ shall have small absolute values, or they shall be suitably positioned. 

The following table gives the norms of related errors under (\ref{newiter}) with $k=2$ for several iterates. Obviously, the error in $\mb{x}^{(m)}$ is ultimately decreased at each stage by the factor $|\lambda_1^2|=0.81$:
\[\begin{array}{|c|c|c|c|c|c|c|c|}\hline
	m&4&5&6&7&8&9&10 \\ \hline
	|\epsilon^{(m)}|&0.884& 0.654& 0.262& 0.391& 0.298& 0.173& 0.189\\ \hline
\end{array}\]

The computation of $\mb{y}^{(m)}$ given by (\ref{main}) requires some additional information. Recall that the associated matrix $\widetilde{M}$ is an $n\times n$ matrix with the same set of eigenvalues as $M$, but the eigenvectors are switched with their complex conjugate counterparts. Consider 
\[P=\begin{bmatrix}
	-2&3&1&-1\\-1/2&1&1/2&-3/4\\0&-1&0&1/2\\1/2&0&-1/2&-1/4
\end{bmatrix}\]
Then we have
\[
D=P^{-1}MP=\left[\begin{array}{cccc}
	9/10&0&0&0\\
	0&2/5+7/10i&0&0\\
	0&0&2/5-7/10i&0\\
	0&0&0&-1/2
\end{array}\right]
\]
We set $\widetilde{M} = P \overline{D} P^{-1}$, and pick a vector $\widetilde{\mb{g}}$ that satisfies the relation $\widetilde{M}\mb{x} + \widetilde{\mb{g}}=\mb{x}$. This is all we need to apply the generalized Chebyshev acceleration. The following table gives the norms of related errors under (\ref{main}) applied to (\ref{newiter}) with $k=2$ for several iterates.
 
\[\begin{array}{|c|c|c|c|c|c|c|c|}\hline
	m&4&5&6&7&8&9&10 \\ \hline
	|\eta^{(m)}|&0.580& 0.946& 0.739& 0.129& 0.204& 0.0861& 0.0274\\ \hline
\end{array}\]

We observe that the basic iterative method is better for the first few iterations. However, the semi-iterative method eventually gives better approximations. See Figure~\ref{fig:error}. The generalized Chebyshev acceleration provides an improvement since it gives better approximations with fewer iterations. 
\begin{figure}[htbp]
	\centering
	\includegraphics[scale=0.8]{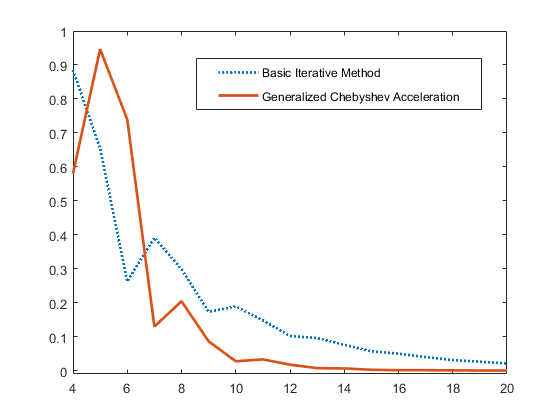}
	\caption{The comparison of the errors $\lVert\epsilon^{(m)}\rVert$ and  $\lVert\eta^{(m)}\rVert$.}
	\label{fig:error}
\end{figure}

To analyze the ultimate behavior of the error in the semi-iterative case, we generalize the treatment in \cite[7.6]{foxparker}. Note that the function $\varphi_1(\theta_1,\theta_2)$ takes real values if $\theta_1=\theta_2$. Let $\alpha$ be a complex number such that $100/81=1/\lambda_1^2=\varphi_1 (\alpha/(2\pi i),\alpha/(2\pi i))$. In our case, we can pick 
$\alpha \approx 0.816$. Expressing (\ref{main}) in terms of the function $\varphi_1$, and focusing on the dominant terms, we see that this iterative process is ultimately equivalent to the following 
\begin{align*}\label{main-v2}		
	\mb{y}^{(m)}=\ &(1+e^{-\alpha}+e^{-2\alpha})\left( M\mb{y}^{(m-1)} + \mb{g} \right) \\
	& -(e^{-\alpha}+e^{-2\alpha}+e^{-3\alpha})\left( \widetilde{M}\mb{y}^{(m-2)} + \widetilde{\mb{g}} \right)\\
	& +e^{-3\alpha}\mb{y}^{(m-3)}, \quad m\geq 3,
\end{align*}
Let us set $a=1+e^{-\alpha}+e^{-2\alpha}$, $b=-(e^{-\alpha}+e^{-2\alpha}+e^{-3\alpha})$, and $c=e^{-3\alpha}$ for the further steps. Note that the error vector satisfies the iteration 
\begin{equation*}\label{main-v4}
	\eta^{(m)} = aM\eta^{(m-1)} + b\widetilde{M}\eta^{(m-2)} +  c\eta^{(m-3)},
\end{equation*}
which can expressed in the following matrix form
\[\left[\begin{array}{c}
	\eta^{(m)}\\\eta^{(m-1)}\\\eta^{(m-2)}
\end{array}\right] = \left[\begin{array}{ccc}
	aM& b\widetilde{M}&c\\\mb{I}&\mb{0}&\mb{0} \\ \mb{0}&\mb{I}&\mb{0}
\end{array}\right] \left[\begin{array}{c}
	\eta^{(m-1)}\\ \eta^{(m-2)}\\ \eta^{(m-3)}
\end{array}\right].\]
The ultimate convergence depends on the spectral radius of the matrix on the right-hand side. It is easy to see that the eigenvalues $\mu$ of this matrix satisfy the equation
\begin{equation}\label{polyeqn}
\mu^3- a\lambda\mu^2  -b\bar{\lambda}\mu - c=0,
\end{equation}
where $\lambda$ is an eigenvalue of $M$. The largest $\mu$ occurs for the eigenvalue $\lambda$ whose absolute value is the spectral radius. Indeed, for $\lambda=\lambda_1^2$, we obtain $$|\mu|_\text{max} \approx 0.442.$$ 
We have verified experimentally that the ratio $\lVert\eta^{(m+1)}\rVert / \lVert\eta^{(m)}\rVert$ approaches to the number $0.442$. See Figure~\ref{fig:errorratios}.

\begin{figure}[htbp]
	\centering
	\includegraphics[scale=.8]{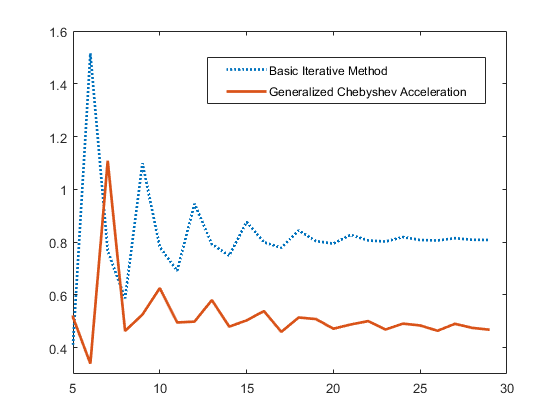}
	\caption{The comparison of  $\dfrac{\lVert\epsilon^{(m+1)}\rVert}{\lVert\epsilon^{(m)}\rVert}$ and $\dfrac{\lVert\eta^{(m+1)}\rVert}{\lVert\eta^{(m)}\rVert}$.}
	\label{fig:errorratios}
\end{figure}

Note that there is an extra multiplication in the semi-iterative method (\ref{main}) compared to (\ref{newiter}). If we have a large sparse matrix, the most consuming part is performing the matrix multiplications. Thus it is reasonable the compare the errors of $2m$th iterate of the basic iterative method with the $m$th iterate of the Chebyhev acceleration. Comparing $0.81^2\approx0.656$ with $0.442$, we see that there is still an advantage using the generalized Chebyshev acceleration instead of the basic iterative method.
\end{example}

\subsection{Several dominant eigenvalues} Now we focus on the remaining cases in which there are several dominant eigenvalues. The following theorem provides some additional cases to which (\ref{newiter}) can be applied. However, the $k$-value can be very large depending on the dominant eigenvalues.
\begin{theorem}\label{exclude}
Let $M$ be a matrix with spectral radius $\rho(M)<1$, and let $S$ be the set of dominant eigenvalues of $M$. Suppose that $|S|\geq 2$. There are two cases:
\begin{enumerate}
	\item If $\lambda_1/\lambda_2$ is a root of unity for each pair of eigenvalues $\lambda_{1}$ and $\lambda_2$ in $S$, then the generalized Chebyshev acceleration is applicable to (\ref{newiter}) for some suitable $k$-value.
	\item If there are two eigenvalues $\lambda_{1}$ and $\lambda_2$ in $S$ such that $\zeta= \lambda_1/\lambda_2$ is not a root of unity, then the generalized Chebyshev acceleration is never applicable to the iterative process (\ref{newiter}).
\end{enumerate}
\end{theorem}
\begin{proof} A complex number $z$ is called a \textit{root of unity} if there exists a positive integer $n$ such that $z^n=1$. Suppose that $\lambda_1/\lambda_2$ is a root of unity for each pair of eigenvalues $\lambda_{1}$ and $\lambda_2$ in $S$. Consider the multiplicative group $G$ generated by the roots of unities $\lambda_1/\lambda_2$ with $\lambda_{1},\lambda_2 \in S$. Let $k_0=|G|$ be the order of this finite group. If all the eigenvalues of $M$ are dominant, then the matrix $M^{k_0}$ has a single eigenvalue, namely $\lambda_1^{k_0}$ for any $\lambda_1 \in S$. In such a case, we can choose $k=k_0$, and the generalized Chebyshev acceleration is applicable to (\ref{newiter}) for this $k$-value.
	
Otherwise, there exists an eigenvalue of $M$ outside the set $S$. We let $\gamma$ be a such an eigenvalue with maximal absolute value. Let $\lambda_1$ be any element in $S$. We have $$\rho(M) = |\lambda_1| > |\gamma| \geq |\lambda|$$
for all eigenvalues $\lambda$ of $M$ outside of $S$. Using the idea in the proof of Theorem~\ref{existingk}, we let $k_1$ be the smallest positive integer such that  $1/3 \geq |\gamma/\lambda_1|^{k_1}$. In this case, we choose $k=k_0k_1$, and the generalized Chebyshev acceleration is applicable to (\ref{newiter}) for this bigger $k$-value.
	
It remains to consider the case for which the generalized Chebyshev acceleration is not applicable. Let $\lambda_{1}$ and $\lambda_2$ in $S$ such that $\zeta= \lambda_1/\lambda_2$ is not a root of unity. The complex numbers $(\lambda_2/\lambda_1)^{k}, k\geq 1$, is never in $\Delta$. Under the generalized Chebyshev polynomials, the iterations of complex numbers outside $\Delta$, such as $(\lambda_2/\lambda_1)^{k}$, tend to infinity. This means that the generalized Chebyshev acceleration applied to (\ref{newiter}) never converges. 
\end{proof}

\section{Normal Sparse Matrices}
In this final section, we analyze the efficiency of our method by applying it to a sparse matrix of large dimension with certain properties. As usual, consider a basic iterative method with an iteration matrix $M$, see (\ref{iter}). Recall that the spectral radius $\rho(M)$ is strictly less than one, and the iterative method converges to the unique solution of the system $(I_n-M)\mathbf{x}=\mathbf{g}$. Before giving our example we first give a heuristic argument for asymptotic convergence, and secondly explain the necessity for the use of normal matrices. 

\subsection{A heuristic argument for asymptotic convergence} If $M$ has real spectrum, and if $\rho(M)=\lambda_1$ for some positive $\lambda_1$, then the classical Chebyshev acceleration has asymptomatic convergence $(1-\sqrt{1-\lambda_1^2})/\lambda_1$.  If $1/\lambda_1 = \cosh(\alpha)=(e^{\alpha}+e^{-\alpha})/2$, then this formula returns $e^{-\alpha}$. See \cite{foxparker}, for the details of this result.

Suppose that the hypothesis of Theorem~\ref{existingk} holds, and for simplicity let us assume that $\rho(M)=\lambda_1$ is real. Following the treatment in \cite{foxparker}, we focus on the unique positive real $\alpha$ such that 
\begin{equation}\label{alpin}
\frac{1}{\lambda_1} = \varphi_1\left(\frac{\alpha}{2\pi i},\frac{\alpha}{2\pi i} \right) = \frac{1}{3}(e^{\alpha}+e^{-\alpha}+1).
\end{equation}
For each eigenvalue $\lambda$ of $M$, we consider the polynomial equation (\ref{polyeqn}). Heuristically, we have found that this polynomial equation has a root with maximal absolute value specifically for $\lambda=\lambda_1$. Moreover, this root turns out to be $e^{-\alpha}$, similar to the classical Chebyshev acceleration. If the equation (\ref{alpin}) is satisfied, we note that $2\lambda_1/(3-\lambda_1) ) = 2/(e^{\alpha}+e^{-\alpha})$. Imitating the computation in the case of classical Chebyshev acceleration, we find in this new case that
$$g(\lambda_1)=\frac{1-\sqrt{1-\left(\frac{2\lambda_1}{3-\lambda_1}\right)^2}}{\left(\frac{2\lambda_1}{3-\lambda_1}\right)} =  \frac{1-\sqrt{1-\left( \frac{2}{e^{\alpha}+e^{-\alpha}} \right)^2}}{\left( \frac{2}{e^{\alpha}+e^{-\alpha}} \right)} = e^{-\alpha}.$$

The iterative process (\ref{iter}) has a single multiplication by $M$. On the other hand, there are two multiplications by $M$ and $\widetilde{M}$  while applying the generalized Chebyshev acceleration (\ref{main}). For sparse matrices of large dimension, this makes a huge difference. To make a fair comparison, we shall compare the quantities $\lambda_1^2$ and $g(\lambda_1)$. See Figure~\ref{fig:faircomp} for this comparison.
\begin{figure}[htbp]
	\centering
	\includegraphics[scale=.5]{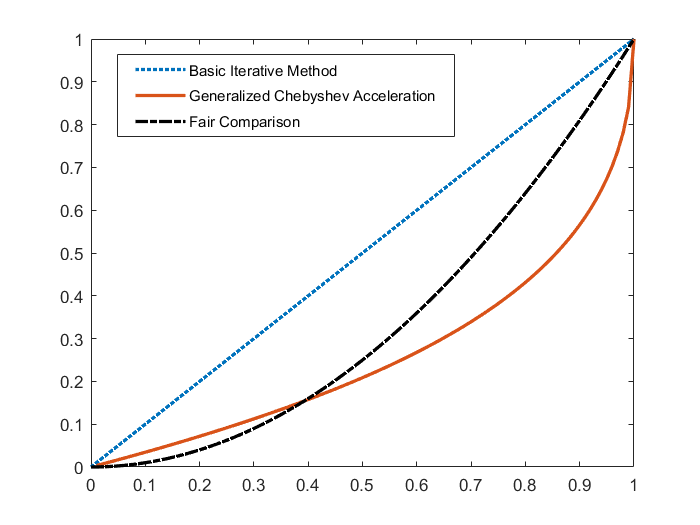}
	\caption{The fair comparison of the basic iterative method and the generalized Chebyshev acceleration for large sparse matrices.}
	\label{fig:faircomp}
\end{figure}

Setting $g(\lambda_1)=\lambda_1^2$, we find that the intersection occurs at $\lambda_1 \approx 0.392$, the unique real root of the polynomial equation $z^3 + z^2 + 2z - 1=0$. This means that if we have an iterative method (\ref{iter}), with spectral radius $\rho(M)$ attained by a unique real eigenvalue $\lambda_1$, then applying the generalized Chebyshev acceleration is practical only if  $\lambda_1$ is bigger than $0.392$.

This computation can be naturally generalized to (\ref{newiter}), and its accelerated counterpart. Computing $\sqrt[k]{0.329}$ for $k\geq 1$, we obtain lower bounds for $\lambda_1$ as follows:
\begin{equation}\label{feas}
\begin{array}{|c|c|c|c|c|c|c|c|}\hline
	k&1&2&3&4&5&6&7\\ \hline
	\lambda_1 \geq&0.392&0.626&0.732&0.791&0.829&0.855&0.874\\\hline
	%	\lambda_2 \leq &0.130&0.361&0.507&0.601&0.665&0.712&0.747\\\hline
\end{array}
\end{equation}
For each positive integer $k$, applying the generalized Chebyshev acceleration to (\ref{newiter}) is practical only if  $\lambda_1$ is bigger than the indicated constant. Recall that we aim to choose a suitable $k$-value in the light of Theorem~\ref{existingk}. By this consideration, the condition $(\lambda/\lambda_1)^k \leq 1/3$ shall be satisfied for each eigenvalue $\lambda$ of $M$, as well. See (\ref{k-ratios}) for a list of the values $\sqrt[k]{1/3}$.

\subsection{The use of normal matrices}
Recall that there are two restrictions while applying the generalized Chebyshev semi-iterative method: the positions of the quotients $\lambda/\lambda_1$, and the computation of $\widetilde{M}$ and $\widetilde{\mb{g}}$. In this paper, we have provided a method that relaxes the former condition by using an alternative iterative process (\ref{newiter}). However, the computation of $\widetilde{M}$ and $\widetilde{\mb{g}}$ remains difficult, especially for arbitrary sparse matrices of large dimension. 

We shall restrict our attention to a case in which the computation of $\widetilde{M}$ is manageable. For this purpose, we suppose that $M$ is normal. Recall that a matrix is called normal if it commutes with its conjugate transpose. For a normal matrix $M$, the conjugate transpose is analogous to the complex conjugate, and we can pick $\widetilde{M} = M^*$. The spectral theorem states that a matrix is normal if and only if it is unitarily similar to a diagonal matrix.

We now explain how to produce a normal sparse matrix $M$ randomly. We first choose a diagonal matrix $D$ with nonzero diagonal entries. Secondly, we choose a random unitary matrix $U_0$ of relatively small dimension and extend this matrix to the same size as $M$ by putting ones on the diagonal, and zeros elsewhere. We set $U=PU_0$ where $P$ is a random permutation matrix with the same size as $M$. We set $M=U^*DU$, and pick $\widetilde{M} = M^* = UDU^*$. We also set $(I_n-\widetilde{M})\mathbf{x}=\widetilde{\mb{g}}$. This is all we need to apply the iteration (\ref{main}). In each iteration, the generalized Chebyshev semi-iterative method requires three scalar multiplications, four vector additions, and the multiplication of two sparse matrices, namely $M$ and $\widetilde{M}$, with certain column vectors. For a sparse matrix of reasonable size, these operations are performed very fast.
	
\begin{example} In this example, we apply the generalized Chebyshev semi-iterative method to a normal sparse matrix of dimension $1000\times 1000$ randomly constructed as described above. This numerical experiment was carried out using MATLAB (R2024b), which allows efficient sparse matrix operations.
	
We first form a diagonal matrix $D$, with nonzero diagonal entries, that has spectral radius $\rho(D)=0.9$. This spectral radius is attained by a unique maximal real eigenvalue $\lambda_1=0.9$, and the other eigenvalues are chosen randomly in the unit disc $|z|\leq \lambda_2=0.6$. This can be done by using the function $a\cdot\exp(2\pi i b)$ with random pairs $(a,b) \in [0,1]\times [0,1]$, and multiplying the resulting complex numbers by $\lambda_2$. For these eigenvalues, the hypothesis Theorem~\ref{existingk} is satisfied with $k=3$. Note that $\lambda_1 \geq 0.732$, and therefore applying the generalized Chebyshev acceleration is practical by (\ref{feas}).  Figure~\ref{fig:eigen} shows the positions of the complex numbers $(\lambda/\lambda_1)^k$ for $k=1, 3$.

\begin{figure}[htbp]
	\centering
	\includegraphics[scale=.5]{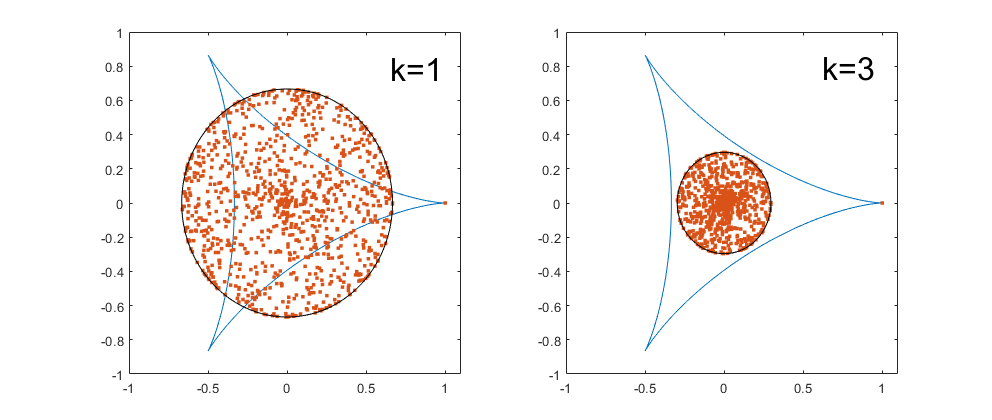}
	\caption{The positions of the quotients $(\lambda/\lambda_1)^k$}
	\label{fig:eigen}
\end{figure}

Secondly, we choose a random unitary matrix $U_0$ of relatively small dimension, say $100\times 100$, and extend this matrix to the same size as $M$ by putting ones on the diagonal, and zeros elsewhere. We set $U=PU_0$ where $P$ is a random permutation matrix with the same size as $M$, and set $M=U^*DU$. The sparse matrix $M$ formed in this fashion, and its conjugate transpose $M^*$, have around 10000 nonzero entries. We choose $\mathbf{x}$ to be the vector whose components are all equal to one, and create $\mb{g}$ and $\widetilde{\mb{g}}$, accordingly. 

We record the norms of the errors attached to the standard iteration (\ref{iter}), and the semi-iterative method (\ref{main}), as well as their ratios $\lVert\epsilon^{(m+1)}\rVert/\lVert\epsilon^{(m)}\rVert$, and $\lVert\eta^{(m+1)}\rVert/\lVert\eta^{(m)}\rVert$, i.e. rates of convergence, for the first 20 iterates. See Figure~\ref{fig:simultaneouscomp}.
\begin{figure}[htbp]
	\centering
	\includegraphics[scale=.4]{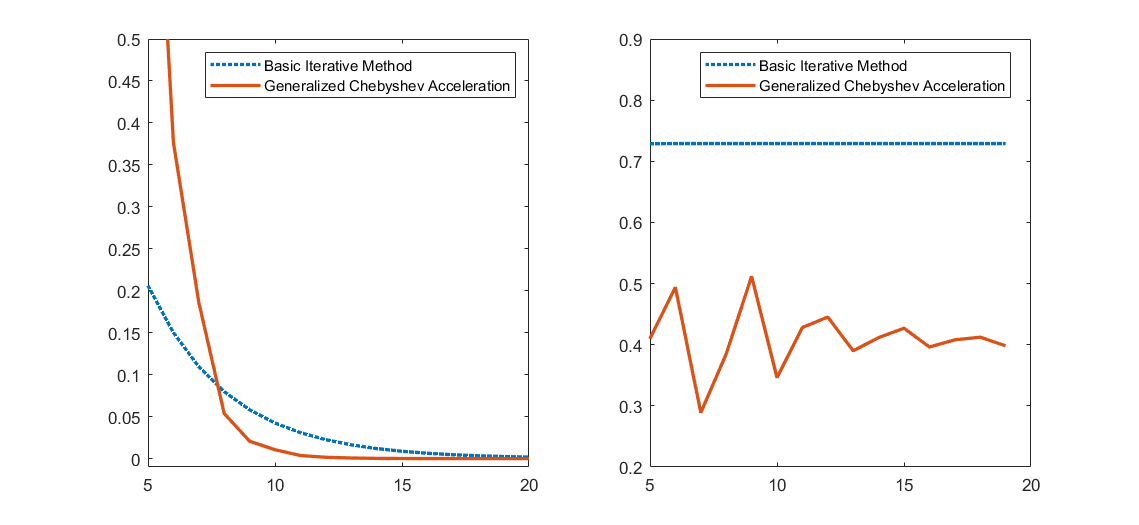}
	\caption{The comparison of the norms of errors and the rates of convergence, respectively.}
	\label{fig:simultaneouscomp}
\end{figure}
We observe that the basic iterative method is better for the first few iterations, but the semi-iterative method eventually gives better approximations. The ultimate rate of convergence, attached to the standard iteration (\ref{newiter}), is $\lambda_1^3=0.729$. This agrees with the experimental data. Moreover, we have $g(\lambda_1^3) = 0.363$. The asymptotic rate of convergence of this example turns out to be slightly bigger than $g(\lambda_1^3)$, at least for the first few iterates, but still smaller than the fair comparison value $(\lambda_1^3)^2=0.531$. 
\end{example}

\section{Acknowledgement}
The author thanks Ö. Küçüksakallı for some useful discussions about group theoretical concepts included in the proof of Theorem~\ref{exclude}.

{\small
\def\refname{References}
\newcommand{\etalchar}[1]{$^{#1}$}

\end{document}